\documentclass[a4paper,leqno]{amsart}

\usepackage{amsmath}
\usepackage{amsthm}
\usepackage{amssymb}
\usepackage{amsfonts}
\usepackage[applemac]{inputenc}
\usepackage{mathrsfs}
\usepackage{pdfsync}
\usepackage{graphicx}
\usepackage{color}
\usepackage{mathabx}

\def\C{{\mathbb C}}% complex numbers
\def\R{{\mathbb R}}% real numbers
\def\N{{\mathbb N}}% nonnegative integers
\def\e{{\varepsilon}}% epsilon
\def\eps{{\varepsilon}}% epsilon
\def\Eq#1#2{\mathop{\sim}\limits_{#1\rightarrow#2}}

\def\le{\leqslant}% lessorequal
\newcommand{\re}{\mathrm{Re}}

\newcommand\norm[1]{\left\lVert#1\right\rVert}

\DeclareMathOperator{\Tr}{Tr}

\theoremstyle{plain}
\newtheorem{theorem}{Theorem}[section]
\newtheorem{lemma}[theorem]{Lemma}
\newtheorem{corollary}[theorem]{Corollary}
\newtheorem{proposition}[theorem]{Proposition}
\newtheorem{hyp}{Assumption}[section]

\theoremstyle{definition}

\newtheorem{remark}[theorem]{Remark}
\newtheorem*{remark*}{Remark}

\numberwithin{equation}{section}

\begin{document}

\title[Semiclassical wave-packets with rotation]
{Semiclassical wave-packets for weakly nonlinear Schr\"odinger equations with rotation}

\author[X. Shen, C. Sparber]{Xiaoan Shen, Christof Sparber}

\address[X. Shen]
{Department of Mathematics, Statistics, and Computer Science, M/C 249, University of Illinois at Chicago, 851 S. Morgan Street, Chicago, IL 60607, USA}
\email{xshen30@uic.edu}

\address[C.~Sparber]
{Department of Mathematics, Statistics, and Computer Science, M/C 249, University of Illinois at Chicago, 851 S. Morgan Street, Chicago, IL 60607, USA}
\email{sparber@uic.edu}

	\begin{abstract}
	We consider semiclassically scaled, weakly nonlinear Schrödinger equations with external confining potentials and additional angular-momentum rotation term. 
	This type of model arises in the Gross-Pitaevskii theory of trapped, rotating quantum gases.
	We construct asymptotic solutions in the form of semiclassical wave-packets, which are concentrated in both space and in frequency 
	around an classical Hamiltonian phase-space flow. The rotation term is thereby seen to alter this flow, but not the corresponding classical action. 
	\end{abstract}
	
\date{\today}


\thanks{C.~Sparber has been supported by the MPS Simons foundation through award no. 851720}
	
\maketitle

%~~~~~~~~~~~~~~~~~~~~~~~~~~~~~~~~~~~~~~~~~~~~~~~~~~~~~~~~~~~~~~~~~~~~~~~~~~~~~~~~~~~~~~~~~~~~~~~~~~~~~~~~~

\section{Introduction}\label{sec:intro}

Semiclassical wave-packets are a well-known tool in the approximate description of quantum mechanics as $\eps \simeq \hbar \to 0$. The latter represents a singular 
limiting regime which leads to highly oscillatory solutions in the corresponding Schr\"odinger dynamics, cf. \cite{Car} for a general introduction. 
In an attempt to overcome this issue, one seeks 
a representation for the exact quantum mechanical wave function $\psi^\eps$ via
    \begin{equation}\label{eq:approx}
    	\psi^{\e}(t,x) \Eq\eps0 \e^{-d/4} v\Big(t,\frac{x-q(t)}{\sqrt{\e}}\Big)e^{i(S(t)+p(t)\cdot (x-q(t)))/\e},\quad x\in \R^d.
    \end{equation}
Here, $q(t) \in \R^d$ and $p(t)\in \R^d$ denote the mean position and momentum at time $t\in \R$, whereas $S(t)\in \R$ is a purely 
time-dependent phase proportional to the classical action. Finally, the amplitude function $v(t, y)\in \mathbb C$ describes slowly varying 
changes due to dispersive effects within the dynamics. The right hand side of \eqref{eq:approx} corresponds to a wave
function which is well-localized (at scale $\sqrt\eps$) both in space and in frequency (or momentum). In particular, for wave functions of the form \eqref{eq:approx} 
the following three quantities
\begin{equation*}
  \|\psi^\eps(t)\|_{L^2(\R^d)},\quad \left\|\Big(\sqrt\eps
    \nabla-i\frac{p(t)}{\sqrt\eps}\Big) 
  \psi^\eps(t)\right\|_{L^2(\R^d)},\quad \text{and }
\quad \left\|\frac{x-q(t)}{\sqrt\eps}
  \psi^\eps(t)\right\|_{L^2(\R^d)}
\end{equation*}
are all of order $\mathcal O(1)$, as $\eps\to 0$.  

It is known that the amplitude $v$ satisfies a homogenized, i.e., $\eps$-independent, Schr\"odinger-type equation 
with an effective quadratic potential (see the derivation below). A popular ansatz for the solution $v$ of this homogenized equation is that of a Gaussian function, 
which can be shown to be propagated exactly. In this case, wave-packets of the form \eqref{eq:approx} comprise semiclassically scaled 
coherent states which minimize the Heisenberg uncertainty relation, cf. \cite{RoCo}.  
Coherent states allow to approximately describe the full quantum dynamics of $\psi^\eps$ via a system of 
ordinary differential equations for $q$, $p$, and the matrices used to parametrize the Gaussian function $v$, see \cite{Ha, He}.
This makes this type of approximation particularly interesting for numerical simulations, where one seeks to represent 
general states $\psi^\eps(t,\cdot)\in L^2(\R^d)$ by a (well chosen) superposition of such Gaussian wave-packets, see \cite{LaLu, Zhou} for a general overview of this topic. 
One should note, however, that the derivation of the aforementioned ordinary differential equations can be rather involved, in particular if the 
Hamiltonian operator governing the dynamics of $\psi^\eps$ has a complicated expression, cf. \cite{CoRo}. 

Motivated by the mean-field description of trapped rotating quantum gases, the aim of this paper is 
to show how to use semi-classical wave-packets in the context of (weakly nonlinear) Schr\"odinger equations with external scalar potential and additional 
angular momentum rotation term. To this end, we consider the following Gross-Pitaevskii equation with rotation
\begin{equation}\label{eq:GP}
    	i\e\partial_{t} \psi^{\e}=-\frac{\e^2}{2}\Delta \psi^{\e}+V(x)\psi^{\e}+\lambda \eps^\alpha |\psi^\eps|^2 \psi+\e(\Omega\cdot L)\psi^{\e}, \quad \psi^{\e}_{\vert{t=0}}=\psi_{0}^{\e}.
\end{equation}
Here $(t,x)\in \R\times \R^d$ with $d=2$ or $3$, respectively, $\lambda \in \R$ denotes a coupling constant, which allows for both focusing and defocusing nonlinearities, and $\alpha = \alpha(d)>0$ is a 
parameter used to ensure the critical strength of the nonlinearity (see below). The operator $\Omega\cdot L$ describes the rotation around a given axis $\Omega\in \R^d$, where 
    $$L=-i x\wedge \nabla, $$
 is the quantum mechanical angular momentum operator. In addition, $V(x)$ denotes some external potential, for which we shall impose:  
   
    \begin{hyp}\label{asm1}
    	The potential $V\in C^{\infty}(\R^d;\R)$ is smooth and sub-quadratic, i.e.,
        $$\partial_x^{\alpha}V\in L^{\infty}(\R^d)\quad \forall |\alpha|\geq 2.$$
    \end{hyp}

A typical example, for such a potential is that of a harmonic confinement, i.e., $V(x) = \frac{1}{2}|x|^2$, which is often used to describe the electromagnetic trapping of 
experimental Bose-Einstein condensates. 

We further assume that the initial data $\psi_{0}^{\e}$ is given in the form of a localized wave-packet, i.e.
    \begin{equation}\label{eq:ini}
    	\psi_{0}^\e(x)= \e^{-d/4}v_0\Big(\frac{x-q_0}{\sqrt{\e}}\Big)e^{i(x-q_0)\cdot p_0/\e},\quad q_0,p_0\in \R^d,
    \end{equation}   
where $v_0\in \Sigma^3$, but not necessarily Gaussian. Here, and in the following we shall denote the natural energy space associated to \eqref{eq:GP} by
 \begin{equation}\label{eq:espace}
    \Sigma^k=\Bigl\{f\in L^2(\R^d):\quad \|f\|_{\Sigma^{k}}:=\sum_{|\alpha|+|\beta|\le k}\|x^{\alpha}\partial_{x}^{\beta}f\|_{L^2(\R^d)}<\infty \Big\}.
    \end{equation}

In the next two sections we shall show how to rigorously derive the dynamical equations needed to 
construct an approximation of the solution to \eqref{eq:GP} in the form \eqref{eq:approx}, first in the linear case $\lambda =0$, and second in the 
case of a critical nonlinearity, i.e., $\lambda\not =0$ and $\alpha= \alpha_{\rm crit}$, where
\[
\alpha_{\rm crit} = \begin{cases}
2 \quad \text{for $d=2,$}\\
\frac{5}{2} \quad \text{for $d=3.$}
\end{cases}
\]
The assumption $\alpha= \alpha_{\rm crit}$ thereby ensures that nonlinear effects are present in the dynamics of the slowly varying amplitude $v$, while for $\alpha>\alpha_{\rm crit}$ the 
problem becomes effectively linearizable. Of course, $\alpha_{\rm crit}>0$ means that the nonlinearity in \eqref{eq:GP} formally vanishes in the limit 
$\eps \to 0$, which is why we call it a weakly nonlinear regime. As we shall see, the approximation in the 
linear case will hold up to Ehrenfest time-scales $t\sim \mathcal O(\ln \frac{1}{\eps})$, whereas in the (weakly) nonlinear case, we will 
need to restrict ourselves to time-scales $t\sim \mathcal O(1)$. In the case $\lambda =0$ and $v_0$ explicitly given by a Gaussian, we shall show how to adapt 
the system of ordinary differential equations governing such wave packets to the case with rotation.

%Before going into more detail, we want to recall the well-known fact that the linear part of the Hamiltonian appearing in \eqref{eq:GP}, can be rewritten in the form
%\begin{equation}\label{eq:magn_op}
%H=-\frac{\eps^2}{2}\Delta + V(x) + \eps \Omega \cdot L  = \frac{1}{2} (-i \eps \nabla - A(x))^2 + V(x) - \frac{|\Omega|^2}{2}r^2,
%\end{equation}
%where $A (x)= x\wedge \Omega $ and $r= \frac{|x \wedge \Omega|}{ |\Omega|}$ denotes the radial distance perpendicular to $\Omega \in \R^d$. Here, 
%$A(x)$ can be considered as the vector potential corresponding to a constant magnetic field. While it is known that semiclassical wave packets can also 
%be developed in the case of general magnetic Schr\"odinger operators, see \cite{Zhou}, the fact that we have this specific form 
%of an angular momentum term, allows for a much more direct and easier development of the approximation, than if one were to try to specify the 
%general magnetic case to the one given by \eqref{eq:magn_op}.

%~~~~~~~~~~~~~~~~~~~~~~~~~~~~~~~~~~~~~~~~~~~~~~~~~~~~~~~~~~~~~~~~~~~~~~~~~~~~~~~~    

\section{The linear case}\label{sec:linear}

 In this section, we shall study the case $\lambda=0$, i.e., we consider
    \begin{equation}\label{lnls}
    	i\e\partial_{t}\psi^{\e}=-\frac{\e^2}{2}\Delta \psi^{\e}+V(x)\psi^\e+\eps(\Omega\cdot L)\psi^\e, \quad \psi^{\e}_{\vert{t=0}} = \psi_0^\eps,
    \end{equation}
 where $\psi_0^\eps$ is given in the form \eqref{eq:ini}. In order to understand how the rotation term influences the dynamics, 
 we first notice that the linear Hamiltonian $H$ can be seen as the 
$\eps$-quantization of the following classical Hamiltonian phase-space function $H:\R^{2d}\rightarrow \R$:
    $$H(x,\xi)=\frac12|\xi|^2+V(x)+\Omega\cdot (x\wedge \xi).$$ 
The corresponding Hamiltonian trajectories for a particle with position $q(t)\in \R^d$ and momentum $p(t) \in\R^d$ are therefore given by
    \begin{equation}\label{traj}
    \begin{cases}
    	\dot{q}=\nabla_p H(q,p)=p+\Omega\wedge q,\quad q(0) = q_0,\\
    	\dot{p}=-\nabla_q H(q,p)=-\nabla V(q)+\Omega\wedge p,\quad p(0)=p_0.
    \end{cases}
    \end{equation}
    \begin{lemma}[Classical dynamics]\label{lem:dyn}
    	Let $(q_0, p_0)\in \R^d\times \R^d$ and $V$ satisfy Assumption \ref{asm1}. Then, \eqref{traj} has a unique global, smooth solution $(q,p)\in C^{\infty}(\R;\R^d)^2$, which grows at most exponentially.
    \end{lemma}
    \begin{proof}
    	The local well-posedness of the solution can be inferred from the fact that $V$ is smooth. From \eqref{traj} we see that $q$ solves the following ordinary differential equation:
    	$$\ddot{q}=-\nabla V(q)+2\Omega\wedge\dot{q}-\Omega\wedge(\Omega\wedge q).$$
    	Multiply both sides by $\dot{q}$,
    	$$\frac{d}{dt}H(q,p)\equiv\frac{d}{dt}\Big(\frac{1}{2}|\dot{q}|^2-\frac{1}{2}|\Omega\wedge q|^2+V(q)\Big)=0.$$
    	We can see that $|\Omega\wedge q|^2\lesssim \langle q\rangle^2,$ and $V(q)\lesssim \langle q\rangle^2$ by Assumption \ref{asm1}, so
    	$$\dot{q}\lesssim \langle q \rangle,$$
    	which shows that 
    	$$|q(t)|\lesssim e^{c_0t}.$$
    	Plugging this into \eqref{traj} yields the same estimate for $p(t)$. 
    \end{proof}

\begin{remark}
The system \eqref{traj} has already been studied in \cite{ArNeSp} in the case of a purely harmonic, but not necessarily isotropic, 
confinement potential $V(x)=\sum_{j=1}^d \gamma_j x_j^2$. It is shown that in this case there are indeed initial data for which the 
solution grows exponentially forward or backward in time, and thus the classical dynamics is no longer trapped within a bounded phase-space region.  
\end{remark} 

Next, we compute the Lagrangian $L(q,p)$ corresponding to $H(q,p)$, via
    	\begin{align*}
    		L(q,p)&=p\cdot \dot{q}-H(q,p)\\
    		&=p\cdot (p+\Omega\wedge q)-\frac12 |p|^2-V(q)-\Omega\cdot (q\wedge p)\\
    		&=\frac12 |p|^2-V(q),
        \end{align*}
        using the fact that $\Omega\cdot(q\wedge p)=p\cdot(\Omega\wedge q)$. One can see that $L(q,p)$ is indeed of the same form as in the case without rotation. 
        In particular, we shall define the associated action function to be as usual, i.e.,
        \begin{equation}\label{act}
        	S(t)=\int_{0}^t \frac{1}{2}|p(s)|^2-V(q(s))\, ds.
        \end{equation}
        The latter will be used to determine the purely $t$-dependent part of the phase of our wave-packets.

    \begin{remark}
    	An alternative way to express the Hamiltonian dynamics with rotation is to introduce the canonical momentum $\pi(t):=p(t)+\Omega\wedge q(t)$, and compute
    	\begin{align*}
    		\dot{\pi}&=\dot{p}+\Omega\wedge \dot{q}\\
    		&=-\nabla V(q)+\Omega\wedge p+\Omega\wedge \pi\\
    		&=-\nabla V(q)+2\Omega \wedge \pi -\Omega\wedge(\Omega\wedge q).		
    	\end{align*}
Using this, \eqref{traj} can be rewritten in the following form:
    	\begin{equation}
    		\begin{cases}
    			\dot{q}=\pi,\\
    			\dot{\pi}=-\nabla V(q)+2\Omega\wedge \pi-\Omega\wedge(\Omega\wedge q).
    		\end{cases}
    	\end{equation} 
    This system has been used to describe rotating solutions of mean-field models for self-gravitating classical particles, see \cite{AS}. 
    \end{remark}

  Having derived the classical dynamics in the case with rotation, we can now turn to the derivation of the semiclassical approximation. To this end, we first change the 
  unknown $\psi^{\e}$ into a new function $u^{\e}$, via
    \begin{equation}\label{eq:resc}
    	\psi^{\e}(t,x)=\e^{-d/4}u^{\e}\Big(t,\frac{x-q(t)}{\sqrt{\e}}\Big)e^{i(S(t)+p(t)\cdot (x-q(t)))/\e},
    \end{equation}
    where $q(t)$ and $p(t)$ are solutions to \eqref{traj}, and $S(t)$ is defined by \eqref{act}. Plugging this ansatz into equation \eqref{lnls} 
    and assuming sufficient smoothness, we obtain, after some lengthy computations, that
    \begin{align*}
   0= & \, 	i\e\partial_{t}\psi^{\e}+\frac{\e^2}{2}\Delta \psi^{\e}-V(x)\psi^{\e}-\eps (\Omega\cdot L)\psi^{\e}\\
   = & \, \e^{-d/4}e^{i\phi/\e}\big(\e R_1+\e^{1/2}R_2+R_3\big),
    \end{align*}
    where we denote 
    $$\phi(t,x)=S(t)+p(t)\cdot(x-q(t)),$$ 
    and we also find
    $$R_1=i\partial_{t}u^{\epsilon}+\frac{1}{2}\Delta u^{\epsilon}-(\Omega\cdot L)u^{\epsilon},$$
    $$R_2=i\nabla u^{\epsilon}\cdot (p-\dot{q}+\Omega\wedge q),$$
    $$R_3=\Big(-\sqrt{\epsilon}y\cdot(\dot{p}-\Omega\wedge p)+p\cdot \dot{q}+V(q)-V(q+\sqrt{\epsilon}y)-|p|^2-\Omega\cdot(q\wedge p)\Big)u^{\epsilon}.$$
  In here, the fact that $S(t)$ is given by \eqref{act} is essential. Recalling that $p$, $q$ are assumed to be solutions to \eqref{traj}, we see that, indeed, $R_2\equiv0$, whereas $R_3$ simplifies to
    	$$R_3=u^{\e}\Big(\sqrt{\e}y\cdot \nabla V(q)+V(q)-V(q+\sqrt{\e}y)\Big).$$
  In order for $u^\eps$ to be a solution to \eqref{lnls}, we therefore have to guarantee that $$\e R_1+R_3=0,$$ which is equivalent to imposing
    \begin{equation}
    	i\partial_{t}u^{\e}=-\frac{1}{2}\Delta u^{\e}+\mathcal V^{\e}(t,y)u^{\e}+(\Omega\cdot L)u^{\e},\quad u^{\e}_{\vert{t=0}}=v_0.
    \end{equation}
    Here, $v_0$ is the initial amplitude induced by \eqref{eq:ini}, and $\mathcal{V}^\e$ is a time-dependent potential given by
    \begin{equation}\label{hess}
    	\mathcal{V}^\e(t,y)=\frac{1}{\e}\big(V(q(t)+\sqrt{\e}y)-V(q(t))-\sqrt{\e}y\cdot \nabla V(q(t))\big).
    \end{equation}
    By formally passing to the limit $\e\to 0$ in this expression, we observe that 
    \[
    \mathcal{V}^\e(t,y)  \Eq\eps0  \frac12y\cdot Q_V(t)y,
    \]
    i.e., a harmonic potential with $Q_V(t)=\nabla^2 V(q(t))$, the Hessian matrix of $V$. Note that Assumption \ref{asm1} implies that  
    $Q_V\in C^\infty_{\rm b}(\R;  \R^{2d})$.
    
    We therefore expect that $u^\eps$ is asymptotically close (in an appropriate norm) to $v$, defined to be the solution of the following, $\eps$-independent 
    effective amplitude equation:
    \begin{equation}\label{vlnls}
    	i\partial_{t}v=-\frac{1}{2}\Delta v+\frac12 \big(y\cdot Q_V(t)y\big) v+(\Omega\cdot L)v, \quad v_{\vert{t=0}}=v_0.
    \end{equation}
    We shall briefly study the existence of solutions $v$ to this equation. The corresponding time-dependent classical Hamiltonian function
    \begin{equation}
    	H(t,y,\xi):=\frac{1}{2}|\xi|^2+\frac{1}{2}y\cdot  Q_V(t)y+\Omega\cdot (y\wedge \xi)
    \end{equation}
    is smooth and sub-quadratic in $(y,\xi)\in \R^{2d}$ and therefore fits within the framework of \cite{Ki}, 
    where the fundamental solution of the associated Schr\"odinger propagator is constructed (see also the appendix). 
    
     \begin{lemma}[from \cite{Ki}]\label{lem:ex}
    	Let $d=2,3$, $v_0\in \Sigma^k$ and $V$ satisfy Assumption \ref{asm1}. Then, for all $k\in \N$, equation \eqref{vlnls} has a unique global solution $v \in C(\R, \Sigma^k)$, satisfying 
	\[\|v(t, \cdot)\|_{L^2(\R^d)}=\|v_0\|_{L^2(\R^d)} \ \forall  t\in \R.\]
	In addition, there exists a $C_{k,d}>0$, such that
	\[\|v(t, \cdot)\|_{\Sigma^k}\lesssim e^{C_{k,d}\, t}. \]
    \end{lemma}
    
   We can now state the main approximation result of this section.
    \begin{proposition}[Linear wave-packets with rotation]
    	Let $d=2$ or $3$, $v_0\in \Sigma^3$ and $V$ satisfy Assumption \ref{asm1}. Consider the semiclassical wave-packet given by
	\begin{equation*}\label{varphi}
    	\varphi^{\e}(t,x)=\e^{-d/4}v\Big(t,\frac{x-q(t)}{\sqrt{\e}}\Big)e^{i(S(t)+p(t)\cdot(x-q(t)))/\e},
    \end{equation*}
    where $v \in C(\R, \Sigma^3)$ is a solution to \eqref{vlnls}, $S(t)$ is the action defined in \eqref{act} and $(q,p)\in C^{\infty}(\R;\R^d)^2$ solve the 
    Hamiltonian equations \eqref{traj}. Then, there exists a constant $C>0$ and independent of $\e\in (0,1]$, such that 
    	\begin{equation*}
    		\|\psi^{\e}(t,\cdot)-\varphi^{\e}(t,\cdot)\|_{L^2(\R^d)}\lesssim \sqrt{\e}e^{Ct}.
    	\end{equation*}
    	In particular, there exists $c>0$ independent of $\e$, such that, as $\eps \to 0:$
    	\begin{equation*}
    		\sup_{0\le t\le c\log\frac1\e}\|\psi^{\e}(t,\cdot)-\varphi^{\e}(t,\cdot)\|_{L^2(\R^d)}\longrightarrow 0.
        \end{equation*}
    \end{proposition}
    
    \begin{remark} This result can be seen as a special case of the one derived in \cite{CoRo} for general sub-quadratic Hamiltonians. However, 
    given the physical relevance of the case with angular momentum rotation term, we shall give a complete proof below. In addition, we 
    shall explicitly clarify the connection between the cases with and without angular momentum in Corollary \ref{GWP}.
    \end{remark}
    
    \begin{proof}
    	The proof follows along the same as in \cite{CaFe}. We first notice that since $V$ is smooth and sub-quadratic, a Taylor-expansion shows 
	\begin{equation}\label{eq:delta}
	|\Delta_V^\eps(t,y)| = \Big|\mathcal{V}^\e(t,y)-\tfrac12y\cdot  Q_V(t)y\Big|\lesssim  \sqrt{\e}|y|^3.
	\end{equation}
    	We can define the error term $r^{\e}(t,y)=u^{\e}(t,y)-v(t,y)$, and obtain that $r^{\e}$ solves the following equation
    	\begin{equation}\label{error}
    		i\partial_{t}r^{\e}=-\frac12 \Delta r^{\e}+(\Omega\cdot L)r^{\e}+\mathcal{V}^\e(t,y)r^{\e}+\Big(\mathcal{V}^\e(t,y)-\frac12 y\cdot  Q_V(t)y\Big)v,
    	\end{equation}
	with vanishing initial data $r^{\e}(0,y)=0$. Since the right hand side of \eqref{error} is given by self-adjoint operators acting on $r^\eps$ plus a source term, a standard energy estimate shows that 
    	\[
	\|r^{\e}(t,\cdot)\|_{L^2(\R^d)}\le \int_0^t |\Delta^\eps_V(\tau, y)| \, d\tau \lesssim \sqrt{\e}\int_{0}^t \||y|^3v(\tau,\cdot)\|_{L^2(\R^d)}\,.
	\]
	We note that 
	\[ \| |y|^3v(\tau,\cdot)\|_{L^2(\R^d)} \le \| v (\tau, \cdot) \|_{\Sigma^3} \lesssim e^{C\tau}, \quad \forall \tau \in \R,
	\]
	in view of Lemma \ref{lem:ex}. We therefore obtain
	\[
	\|r^{\e}(t,\cdot)\|_{L^2(\R^d)} \lesssim \sqrt{\e}e^{Ct}.
	\]
    	The result then follows from the fact that the wave-packet rescaling \eqref{eq:resc} leaves the $L^2$-norm invariant, and thus
	\[
	\|\psi^{\e}(t,\cdot)-\varphi^{\e}(t,\cdot)\|_{L^2(\R^d)} = \|u^{\e}(t,\cdot)-v(t,\cdot)\|_{L^2(\R^d)}\equiv \|r^{\e}(t,\cdot)\|_{L^2(\R^d)}\lesssim \sqrt{\e}e^{Ct}.
	\]
    \end{proof}
 
  \begin{remark}
  The time-scale $t\sim \mathcal O(\log\frac1\e)$ is called the Ehrenfest-time. It is known to be the longest possible time-scale until which one can hope to establish an 
  effective semi-classical approximation, in general, cf. \cite{CoRo, SVT}. Under stronger assumptions on $v_0$ it is possible to generalize the above 
  approximation result to hold in stronger $\Sigma^k$-norms up to Ehrenfest-time.
  \end{remark}

   A particular class of global solutions $v$ to \eqref{vlnls} is obtained for (complex-valued) Gaussian initial data. More precisely, by following the ideas of Hagedorn \cite{Ha}, we consider initial data of the form
    \begin{equation}\label{gauss}
    	v_0(y)=\frac{1}{(\det A_0)^{1/2}}\exp\Big(-\frac{1}{2}y\cdot  \big(B_0A_{0}^{-1}\big)y\Big),
    \end{equation}
    where the matrices $A_0$ and $B_0$ satisfy the following properties:
    \begin{equation}\label{mat}
    \begin{aligned}
    	&A_0 \text{ and } B_0 \text{ are invertible;} \\
    	&\text{$B_0A_{0}^{-1}$ is symmetric; $B_{0}A_{0}^{-1}=M_1+iM_2$, with $M_j$ symmetric;}\\
        &\text{$\re  \,B_{0}A_{0}^{-1}$ is strictly positive definite;}\\
        &\text{$(\re \, B_{0}A_{0}^{-1})^{-1}=A_{0}A_{0}^*$.}
    \end{aligned}
    \end{equation}
   We shall now show that such Gaussian wave packets are indeed propagated by equation \eqref{vlnls}.
    \begin{corollary}[Gaussian wave-packets]\label{GWP}
    	Let $v \in C(\R, \Sigma^k)$ be the solution to \eqref{vlnls}, with $v_0$ given by \eqref{gauss}--\eqref{mat}. Then for all time $t\in \R$, $v$ is given by 
    	\begin{equation}\label{gaussv}
    	      v(t,y)=\frac{1}{\big(\det A(t)\big)^{1/2}}\exp\Big(-\frac{1}{2}y\cdot \big(B(t)A(t)^{-1}\big)y\Big),
    	\end{equation}
    	provided $A(t)$ and $B(t)$ solve the following ordinary differential equations
    	\begin{equation}\label{ode}
    		\begin{cases}
    			\dot{A}(t)=iB(t)-[R_{\Omega},A(t)];\quad A(0)=A_0,\\
    			\dot{B}(t)=iQ_V(t)A(t)-[R_{\Omega},B(t)];\quad B(0)=B_0,
    		\end{cases}
    	\end{equation}
    	where $R_{\Omega}$ is a skewed-symmetric matrix, given by 
    	$$R_{\Omega}=
    	\begin{bmatrix}
    		0 & \Omega_{3} & -\Omega_2\\
    		-\Omega_{3} & 0 & \Omega_1\\
    		\Omega_2 & -\Omega_1 & 0 \\
    	\end{bmatrix}.$$
    	In addition, \eqref{ode} guarantees that $A(t)$ and $B(t)$ satisfy \eqref{mat} for all $t\in \R$.
    \end{corollary}
 
    \begin{proof}
    	We first assume that $A$ and $B$ satisfy \eqref{mat} for all $t\in \R$ and plug the Gaussian ansatz \eqref{gaussv} into \eqref{vlnls}. After another lengthy computation we find 
	that $v$ solves \eqref{vlnls}, if and only if the matrices $A$ and $B$ satisfy:
    	\begin{align*}
    		&\Tr(i\dot{A}A^{-1}+BA^{-1})\,+\\
    		&\, y^\top \Big(i\dot{B}A^{-1}-iBA^{-1}\dot{A}A^{-1}-BA^{-1}BA^{-1}+Q_{V}+2iR_{\Omega}BA^{-1}\Big)y=0.
    	\end{align*} 
        In a first step, this implies
        \begin{equation}\label{Aeq}
        	i\dot{A}A^{-1}+BA^{-1}+\Lambda=0,
        	%&y\cdot \Big (i\dot{B}A^{-1}-iBA^{-1}\dot{A}A^{-1}-BA^{-1}BA^{-1}+Q_{V}+2iR_{\Omega}BA^{-1}\Big)y=0,
        \end{equation}
        where $\Lambda$ is any matrix such that $\Tr(\Lambda)=0$. By choosing $\Lambda=i[R_{\Omega},\,A]A^{-1}$ this fact is guaranteed and we directly 
        obtain the first equation of \eqref{ode}. Using the right hand side of this 
        equation as the new expression for $\dot{A}$, we find, in a second step, the following condition for $B$:
        \begin{equation}\label{Beq}
        y^\top \Big(\dot{B}-iQ_{V}A+[R_{\Omega},\,B]\Big)A^{-1}y +y^\top  \Big(BA^{-1}R_{\Omega}+R_{\Omega}BA^{-1}\Big)y =0.
        \end{equation}
    	In fact, since 
    	\begin{align*}
    		 y^\top  R_{\Omega}BA^{-1}y =-(BA^{-1}R_{\Omega}y)^\top y =- y^\top BA^{-1}R_{\Omega}y ,
    	\end{align*}
    	we have $y^{\top}\big(BA^{-1}R_{\Omega}+R_{\Omega}BA^{-1}\big)y=0$, which means that \eqref{Beq} simplifies to 
	\[
	y^\top\Big(\dot{B}-iQ_{V}A+[R_{\Omega},\,B]\Big)A^{-1}y=0,
	\]
	which is guaranteed to hold, provided $B$ satisfies the second equation of \eqref{ode}.
	
	    To prove that $A(t)$ and $B(t)$ satisfy \eqref{mat}, we emplot the same argument as in \cite[Lemma 2.1]{Ha}: We first define two functions
    	$$F(t):=A^*(t)B(t)+B^*(t)A(t), \quad G(t):=A^\top(t)B(t)-B^\top(t)A(t).$$
    	and note that 
    	\begin{align*}
    		\dot{F}(t)=& \, (iB-[R_{\Omega},A])^{*}B+A^*(iQA-[R_{\Omega}, B])\\
    		&\, +(iQA-[R_{\Omega}, B])^*A+B^*(iB-[R_{\Omega},A])=0,
    	\end{align*}
    	as $R_{\Omega}$ is skewed-symmetric and $Q_V(t)$ is symmetric. Hence,
    	\begin{align*}
    		F(t)&=F(0)=A^*_0B_0+B^*_0A_0=A_0^*\Big(B_0A_0^{-1}+(A_0^{*})^{-1}B_0^*\Big)A_0\\
    		&=A_0^*\Big(2\re(B_0A_0^{-1})\Big)A_0=2A_0^*\Big((A_0A_0^*)^{-1}\Big)A_0=2\mathbb I.
    	\end{align*}
    	For any $z\in \C$, we thus have
    	$$\langle z, \, z\rangle=\frac12\langle z,\, F(t)z\rangle=\frac12\langle A(t)z,\, B(t)z\rangle+\frac12 \langle B(t)z,\, A(t)z\rangle, $$
    	which equals zero only if $z=0$. Thus $\ker{A(t)}=\ker{B(t)}=\{0\}$, i.e. $A(t)$ and $B(t)$ are invertible. 
	
	Similarly, we infer that $\dot{G}(t)=0$, and thus $G(t) = G(0)$, where
    	\begin{align*} G(0)&=A_{0}^{\top}\Big(B_{0}A_{0}^{-1}-(A_{0}^{\top})^{-1}B^{\top}_{0}\Big)A_{0}\\
    		&=A_{0}^{\top}\Big(B_{0}A_{0}^{-1}-(B_{0}A_{0}^{-1})^{\top}\Big)A_0=0, 
    	\end{align*}
	since $B_{0}A_{0}^{-1}$ is symmetric. 
        Hence, $A^\top(t)B(t)=B^\top(t)A(t)$, which shows that $B(t)A(t)^{-1}$ is symmetric. Finally, since $F(t)=2\mathbb I$, we have
    	\[
    		2 \mathbb{I}= A^*\Big(BA^{-1}+(A^*)^{-1}B^*\Big)A= A^*\Big(BA^{-1}+\big(BA^{-1}\big)^{*}\Big)A = 2A^*\re\big(BA^{-1}\big)A,
         \]
         and thus $\re\big(BA^{-1}\big)=(AA^*)^{-1}$. This also proves that $\re(BA^{-1})$ is strictly positive definite.
     	\end{proof}
	
    	\begin{remark}
    		One might wonder, why we chose a commutator in equation \eqref{Aeq}, when any other $\Lambda$ with trace equal to zero would also be a possibility. However, the commutators are a natural choice in view of the following fact:
		Let 
		$${A}_\Omega(t)=e^{tR_{\Omega}}A(t)e^{-tR_{\Omega}},\quad {B}_\Omega(t)=e^{tR_{\Omega}}B(t)e^{-tR_{\Omega}},$$
		be two new matrices matrices obtained by conjugating $A$ and $B$ with time-dependent rotation matrices. Then one checks that $A_\Omega$, $B_\Omega$ solve 
		\begin{equation} \label{Hode}   		
    			\begin{cases}
    				\dot{A}_\Omega(t)=i B_\Omega(t)\\
    				\dot{B}_\Omega(t)=i Q_{V, \Omega}(t) A_\Omega(t),
    			\end{cases} 
    		\end{equation}
		which implies that $A_\Omega$, $B_\Omega$ also have all the properties \eqref{mat}. The system \eqref{Hode} is identical to 
		the one originally derived by Hagedorn, provided ${Q}_{V, \Omega} = Q_V$. The latter is true
		for potentials $V$ which are symmetric with respect to the rotation axis $\Omega$, since in this case $[R_\Omega, V]=0$. This reflects 
		the well-known fact that solutions $v$ to (nonlinear) Schr\"odinger equations with 
		angular momentum term are related to solutions $\tilde v$ of the same equation but without angular momentum term, via the following time-dependent unitary transformation
		\begin{equation}\label{eq:tchange}
		{v}_\Omega(t,y)=e^{it\Omega\cdot L}v(t,y)=v\big(t,e^{tR_{\Omega}}y\big),
		\end{equation}
		see \cite{AnMaSp, ArNeSp} for more details. Acting with with change of variables onto the Gaussian ansatz \eqref{gaussv}, one can see that the latter remains Gaussian, provided 
		$A$ and $B$ are replaced by $A_\Omega$ and $B_\Omega$, respectively. 
	   	\end{remark}

    \medskip
%~~~~~~~~~~~~~~~~~~~~~~~~~~~~~~~~~~~~~~~~~~~~~~~~

\section{Extension to the weakly nonlinear case}\label{sec:nonlinear}

In this section, we shall show how to extend the construction of semi-classical wave packets to the case of weakly nonlinear Schr\"odinger equations with rotation. 
We thereby follow the ideas of \cite{CaFe} and only consider the critical case, where $\alpha = 1+\frac{d}{2}$. We consequently are interested in 
\begin{equation}\label{eq:nls}
    	i\e\partial_{t} \psi^{\e}=-\frac{\e^2}{2}\Delta \psi^{\e}+V(x)\psi^{\e}+\lambda \eps^{1+d/2} |\psi^\eps|^2 \psi^{\e}+\e(\Omega\cdot L)\psi^{\e}, \quad \psi^{\e}_{\vert{t=0}}=\psi_{0}^{\e},
\end{equation}
for $d=2$ or $3$, and initial data $\psi_0^\eps$ given in the form \eqref{eq:ini}. Rewriting the unknown $\psi^\eps$ in terms of $u^\eps$ as given by \eqref{eq:resc}, we notice that 
$|\psi^\eps|^2  \sim \eps^{-d/2}$ and thus, by following the same steps as in the linear case, we (formally) 
arrive at the corresponding amplitude equation with cubic nonlinearity, i.e.
 \begin{equation}\label{vnls}
    	i\partial_{t}v=-\frac{1}{2}\Delta v+\frac12 \big(y\cdot Q_V(t)y\big) v + \lambda |v|^2 v +(\Omega\cdot L)v, \quad v_{\vert{t=0}}=v_0.
\end{equation}
\begin{remark}
If we had chosen a subcritical $\alpha >  1+\frac{d}{2}$, the nonlinearity would not appear in \eqref{vnls}, and thus the situation is very similar to the one in our previous section.
The supercritical case $\alpha < 1+\frac{d}{2}$, however, is much more involved and the only rigorous results available to date are for the case of (nonlocal) Hartree nonlinearities, cf. \cite{CC}.
\end{remark}

Equation \eqref{vnls} falls within the class of models studied in \cite{AnMaSp}, and local existence of solutions is guaranteed for smooth initial data. More precisely, we have:

\begin{lemma}[Local Existence]\label{lem:maxexist}
  Let $v_0\in \Sigma^k$ with $k>d/2$. There exists $T_{\rm crit}\in (0,+\infty]$
  and a unique maximal solution $v\in C([0,T_{\rm crit});\Sigma^k)$ to
  \eqref{vnls}, such that $\| v(t,\cdot) \|_{L^2} = \| v_0
  \|_{L^2}$. The solution is maximal  
  in the sense that if $T_{\rm crit}<\infty$, then
  \begin{equation*}
    \lim_{t\to T_{\rm crit}}\|v(t,\cdot )\|_{\Sigma^k}=+\infty. 
  \end{equation*}
\end{lemma}
\begin{remark} 
In general, $T_{\rm crit}<+\infty$, in particular in the focusing case $\lambda <0$ where the appearance of finite-time blow-up is a possibility, see \cite{AnMaSp}. 
The change of variable \eqref{eq:tchange} allows one to map solutions to \eqref{vnls} onto solutions of NLS without rotation term, but with time-dependent, 
sub-quadratic potentials, for which the long time behavior is studied in \cite{Carles}.
\end{remark}

Our main result in this section is as follows:

\begin{theorem}[Weakly nonlinear wave packets with rotation]
Let $d=2$ or $3$, $v_0\in \Sigma^3$ and $V$ satisfy Assumption \ref{asm1}. Let $S$ be the classical action \eqref{act} and 
\[
\varphi^{\e}(t,x)=\e^{-d/4}v\Big(t,\frac{x-q(t)}{\sqrt{\e}}\Big)e^{i(S(t)+p(t)\cdot(x-q(t)))/\e}
\]
be a semiclassical wave-packet concentrated, as before, along the trajectories \eqref{traj}, but 
with an amplitude $v\in C([0,T_{\rm crit});\Sigma^3)$ given by the maximal solution to the nonlinear equation \eqref{vnls}. Then, for any $T<T_{\rm crit}$, we have
\begin{equation*}
    		\sup_{0\le t\le T}\|\psi^{\e}(t,\cdot)-\varphi^{\e}(t,\cdot)\|_{L^2(\R^d)}\lesssim \sqrt{\eps}.
        \end{equation*}
\end{theorem}

Even in cases where the maximal life-span of solutions to \eqref{vnls} is $T_{\rm crit} = +\infty$, it is 
not clear whether this nonlinear approximation result extends up to Ehrenfest times $t\sim \mathcal O( \log \frac{1}{\eps})$. 
With considerably more effort, however, it was shown in \cite{CaFe} that time-scales of order $t\sim \mathcal O(\log \log \frac{1}{\eps})$ 
can be reached in the critical case. Here, we only treat the case of finite, macroscopic times $t\sim \mathcal O(1)$, in the interest of 
giving a short and not too technical proof which relates to the linear case in a transparent way.

\begin{proof}
    As in the linear case, we denote the remainder by $r^{\e}(t,y)=u^{\e}(t,y)-v(t,y)$, and first note that the unknown $u^{\e}$ defined via \eqref{eq:resc} solves
    \begin{equation}
	    i\partial_{t}u^{\e}=-\frac{1}{2}\Delta u^{\e}+\mathcal V^{\e}(t,y)u^{\e}+(\Omega\cdot L)u^{\e}+\lambda|u^{\e}|^2u^{\e},\quad u^{\e}_{\vert{t=0}}=v_0,
    \end{equation}
    where $\mathcal V^{\e}$ is given by \eqref{hess}. We recall the definition of $\Delta^{\e}_{V}(t,y)$ given by \eqref{eq:delta} and 
    consequently infer that $r^{\e}$ is the solution to
	\begin{equation}\label{eq:rnls}
		i\partial_{t}r^{\e}=-\frac{1}{2}\Delta r^{\e}+\mathcal{V}^\e (t,y) r^{\e}+(\Omega\cdot L)r^{\e}+\Delta^{\e}_{V}(t,y)v+\lambda \Big(|u^{\e}|^2u^{\e}-|v|^2v\Big),
	\end{equation}
	subject to initial data $r^{\e}(0,y)=0$.
	 
	Next, we denote by $U_\Omega^\eps(t, s)$ the $L^2$-unitary operator furnishing the Schr\"odinger dynamics associated to the time-dependent Hamiltonian 
	\begin{equation}\label{eq:hamo}
	H_\Omega^\eps(t)=-\frac{1}{2}\Delta+\mathcal{V}^\e(t,x)+(\Omega\cdot L).
	\end{equation} 
	By applying Duhamel's formula to \eqref{eq:rnls}, we obtain
	\begin{align*}
		r^{\e}(t+\tau)=&\, U_\Omega^\eps(\tau, t)r^{\e}(t)-i\int_{t}^{t+\tau}U_\Omega^\eps(t+\tau, s)\Delta^{\e}_{V}v(s)\,ds\\
		&\, -i\lambda\int_{t}^{t+\tau}U_\Omega^\eps(t+\tau, s)\Big(|u^{\e}|^2u^{\e}-|v|^2v\Big)(s)\,ds.
	\end{align*}
	In view of the results described in the Appendix, the propagator $U_\Omega^\eps(t,s)$ allows for $\eps$-independent local in-time dispersive estimates. 
	Recall from \cite{KeTa} that $(q,r)$ is an admissible Strichartz-pair associated to the space-time norm $L^q_tL^r_x$, if $2\le r\le \frac{2d}{d-2}\,(\text{resp. }2\le r< \infty \text{ if $d=2$}),$ and 
	$$\frac{2}{q}=d\Big(\frac{1}{2}-\frac{1}{r}\Big).$$
	Define $I=[t,t+\tau]$, with $t\geq 0$, $\tau>0$, and let
	$$q=\frac8d,\quad r=4,$$
	such that $(q,r)$ is admissible. In addition, we put 
	$$q'=\frac{8}{8-d}, \quad r'=\frac43$$
	being the H\"older conjugates of $(q,r)$. The Strichartz estimates derived in \cite{KeTa} then imply
	\begin{equation*}
		\|r^{\e}\|_{L^{8/d}(I;L^{4})}\lesssim \|r^{\e}\|_{L^2}+\|\Delta^{\e}_{V}v\|_{L^1(I;L^2)}+\norm{|u^{\e}|^2u^{\e}-|v|^2v}_{L^{8/(8-d)}(I;L^{4/3})}. 
	\end{equation*}
	For the last term, we have the following pointwise estimate
	\begin{equation}\label{ptwise}
		\Big||u^{\e}|^2u^{\e}-|v|^2v\Big|\lesssim |r^{\e}|\Big(|r^{\e}|^2+|v|^2\Big).
	\end{equation}
	By H\"older's inequality,
	\begin{equation}\label{ineq1}
	    \begin{aligned}
		   \|r^{\e}\|_{L^{8/d}(I;L^{4})}\lesssim \, &\|r^{\e}\|_{L^2}+\|\Delta^{\e}_{V}v\|_{L^1(I;L^2)}\\
		   &+\Big(\norm{r^{\e}}^2_{L^{8/(4-d)}(I;L^{4})}+\norm{v}^2_{L^{8/(4-d)}(I;L^{4})}\Big)\|r^{\e}\|_{L^{8/d}(I;L^{4})}.
	    \end{aligned}
	\end{equation}
	Since amplitude functions $u^{\e}$ and $v$ solve evolutionary equations within the same class of nonlinear Schr\"odinger type models with smooth and sub-quadratic potentials, 
	Lemma \ref{lem:maxexist} yields that both $u^{\e},v\in C([0,T];\Sigma^{k})$. Hence, we have 
	$$\|Pu^{\e}\|_{L^{\infty}([0,T];L^2)}+\|Pv\|_{L^{\infty}([0,T];L^2)}\le C(T),$$
	for any operator $P\in \{{\rm Id},\, \nabla,\, x\}$. 
	
	Next, we recall the Gagliardo-Nirenberg inequality, i.e.,
	\begin{equation*}
		\|f\|_{L^{4}(\R^{d})}\lesssim \|f\|^{1-d/4}_{L^2(\R^d)}\|\nabla f\|^{d/4}_{L^2(\R^d)},\quad \forall f\in H^{1}(\R^d).
	\end{equation*}
	Applying this to $v$ yields
	$$\|v\|_{L^4}\lesssim \|v\|_{L^{2}}^{2-d/2}+\|\nabla v\|_{L^2}^{d/2}\le C(T),$$
	and same is true for $u^{\e}$. Hence
	\begin{equation}
		\norm{r^{\e}}^2_{L^{8/(4-d)}(I;L^{4})}+\norm{v}^2_{L^{8/(4-d)}(I;L^{4})}\lesssim C(T)\Big(\int_{t}^{t+\tau}\,ds\Big)^{\frac{4-d}{4}}\lesssim \tau^{(4-d)/4} .
	\end{equation}
	Thus \eqref{ineq1} can be reduced to 
	\begin{equation}\label{ineq2}
		\|r^{\e}\|_{L^{8/d}(I;L^{4})}\lesssim \|r^{\e}\|_{L^2}+\|\Delta^{\e}_{V}v\|_{L^1(I;L^2)}+ \tau^{(4-d)/4}\|r^{\e}\|_{L^{8/d}(I;L^{4})}.
	\end{equation}
	Now, fix $\tau < 1$ to be sufficiently small, and repeat this estimate a finite number of times to cover $[0,T]$. This yields
	\begin{equation}
	      \|r^{\e}\|_{L^{8/d}([0,T];L^4)}\lesssim \|r^{\e}\|_{L^1([0,T];L^2)}+\|\Delta^{\e}_{V}v\|_{L^1([0,T];L^2)}.
	\end{equation}
	
	Next, applying Strichartz estimates again, with a second admissible pair $(q_1,r_1)=(\infty, 2)$ on $J=[0,t]$ for $0\le t\le T$,
	\begin{equation}
		\|r^{\e}\|_{L^{\infty}(J;L^2)}\lesssim \|\Delta^{\e}_{V}v\|_{L^1(J;L^2)}+\norm{|u^{\e}|^2u^{\e}-|v|^2v}_{L^{8/(8-d)}(J;L^{4/3})}.
	\end{equation}
	Using the pointwise estimate \eqref{ptwise} and repeating the steps \eqref{ineq1}--\eqref{ineq2}, we obtain
	\begin{align*}
		\|r^{\e}\|_{L^{\infty}(J;L^2)}&\lesssim \|\Delta^{\e}_{V}v\|_{L^1(J;L^2)}+\|r^{\e}\|_{L^{8/d}(J;L^4)}\\
		&\lesssim \|r^{\e}\|_{L^1(J;L^2)}+\|\Delta^{\e}_{V}v\|_{L^1(J;L^2)}\\
		&\lesssim  \|r^{\e}\|_{L^1(J;L^2)}+ \sqrt{\e}\norm{|y|^3v(t,y)}_{L^{1}(J;L^2)},
	\end{align*}
	where the last inequality follows by Taylor expansion, just like in the linear case. The above estimate is readily observed to be of Gronwall type, which consequently yields
	$$\|r^{\e}(t)\|_{L^2}\lesssim \sqrt{\e}e^{t}, \quad t\in [0,T],$$
	and thus
	\begin{equation*}
	    \sup_{0\le t\le T}\|u^{\e}(t,\cdot)-v(t,\cdot)\|_{L^{2}}\equiv \sup_{0\le t\le T}\|r^{\e}(t,\cdot)\|_{L^{2}}\lesssim \sqrt{\e}.
	\end{equation*}
	By recalling the fact that the wave-packet rescaling \eqref{eq:resc} leaves the $L^2$-norm invariant the proof is complete.
\end{proof}

\bigskip

%------------------

\appendix

\section{On the existence of Strichartz estimates}\label{sec:strich}

We briefly discuss the existence of $\eps$-independent Strichartz estimates for the propagator $U_\Omega^\eps (t,s)$ associated to the 
Hamiltonian $H^\eps_\Omega(t)$ as given by \eqref{eq:hamo}. To this end, we first consider the case without rotation $\Omega = 0$, i.e., we 
study solutions 
\[
u^\eps(t,x)  = U^\eps(t,s) u_0(x),
\]
to the non-autonomous Schr\"odinger equation
\begin{equation}\label{eq:nona}
i \partial_t u^\eps +\frac{1}{2} \Delta u^\eps = \mathcal{V}^\e(t,x) u^\eps, \quad u^\eps(s,x) = u_0(x),
\end{equation}
where we recall from \eqref{hess} that 
\[
\mathcal{V}^\e(t,y)=\frac{1}{\e}\big(V(q(t)+\sqrt{\e}y)-V(q(t))-\sqrt{\e}y\cdot \nabla V(q(t))\big).
\]
One seeks to construct a strongly continuous map $(t,s) \mapsto U^\eps(t, s)$ which is unitary on $L^2(\R^d)$, and which satisfies $U^\eps(t,t) = {\rm Id}$,
\[
U^\eps(t, \tau)U^\eps(\tau, s) = U^\eps(t, s), \quad {U^\eps(t, s)}^\ast= {U^\eps(t, s)}^{-1}.
\]

Under our hypothesis on the potential $V$, this can be done using the approach developed in \cite{Fu} (see also \cite{Dam} for a detailed revisit of this technique). 
Indeed, Assumption \ref{asm1} implies that $\mathcal{V}^\e$ is a real-valued $C^\infty$-function, such that $\mathcal{V}^\e\in L^2_t L^\infty_{{\rm loc}, x}$ and, for any fixed $t\in \R$, 
$\mathcal{V}^\e(t,\cdot)$ is sub-quadratic in space. A short computation also shows that 
\[
\quad \forall \, |\alpha|= 2 \, : \, \partial_y^\alpha \mathcal{V}^\e (t,y) =  \partial^\alpha_x V(q(t) + \sqrt{\eps} y),
\]
which in view of Lemma \ref{lem:dyn} and Assumption \ref{asm1} implies that for $t\in [0,T]$:
\begin{equation}\label{eq:M}
M:= \| \nabla_y^2 \mathcal{V}^\e \|_{L^2_tL^\infty_x}=  \| \nabla_y^2 {V} \|_{L^2_tL^\infty_x}\le c(T).
\end{equation}
In particular, $M>0$ is $\eps$-independent. 

The propagator $U^\eps(t,s)$ can thus be constructed, using Fujiwara's time-slicing approach, from an associated family of parametrices given by 
\[
E^\eps(t,s) \varphi(y) = \Big( \frac{1}{2\pi i (t-s)} \Big)^{d/2} \int_{\R^d} e^{i S^\eps(t,s,y, z)} \varphi(z) \, dz, \quad \varphi \in \mathcal S(\R^d),
\]
where $S^\eps$ is the associated classical action, cf. \cite[Chapter 1.5]{Dam}. For small enough $\delta > 0$ and $0<|t-s|<\delta$, this oscillatory integral defines 
a bounded operator 
\[
\|E^\eps(t,s) \varphi\|_{L^2}  \le \gamma \| \varphi\|_{L^2},
\] 
where the constant $\gamma >0$ only depends on $\delta$ and $M$ defined in \eqref{eq:M} above. Thus, also $\gamma>0$ is seen to be $\eps$-independent. Taking a partition of the time-interval $[s,t]$ 
and an associated iterated integral operator induced by $E^\eps(t,s)$, one obtains the propagator $\{ U^\eps(t,s)\, : t,s \in [-T,T]\} $ by taking the size of the partition 
step to zero (in an appropriate sense), cf. \cite[Chapter 1.6]{Dam} for full details. In particular, $U^\eps(t,s)$ inherits the dispersive properties of $E^\eps(t,s)$ in the 
sense that 
\[
\| U^\eps(t,s) \varphi \|_{L^\infty(\R^d)} \le \frac{C}{|t-s|^{d/2}}\| \varphi \|_{L^1(\R^d)},
\]
where $C=C(\delta, M)>0$. The general theory developed in \cite{KeTa} shows that this short-time dispersive estimate is sufficient 
to imply the existence of local in-time Strichartz estimates for the propagator $U^\eps(t,s)$, and thus for the solution $u^\eps$ to \eqref{eq:nona}.

Finally, we note that by applying the time-dependent change of variables \eqref{eq:tchange} to $u^\eps$, i.e., by defining
\[
{u}^\eps_\Omega(t,y)=e^{it\Omega\cdot L}u^\eps(t,y)=u^\eps\big(t,e^{tR_{\Omega}}y\big),
\]
we obtain the solution $u^\eps_\Omega$ to a Schr\"odinger equation with rotation
\begin{equation}
i \partial_t u_\Omega^\eps +\frac{1}{2} u_\Omega^\eps = \mathcal{V}^\e_{\Omega}(t,y) u_\Omega^\eps + \Omega \cdot L u^\eps_\Omega, \quad u_\Omega^\eps(s,x) = u_0(x),
\end{equation}
where $\mathcal{V}^\e_\Omega(t,y) = \mathcal{V}^\e \big(t,e^{tR_{\Omega}}y\big)$. Since $R_\Omega$ is the generator of an orthogonal time-dependent rotation, this change of 
variables leaves every $L^p(\R^d)$-norm of $u^\eps$ invariant and guarantees that $\mathcal{V}^\e_\Omega$ is of the same class as $\mathcal{V}^\e$ itself. In particular, it holds 
\[
M= \| \nabla_y^2 \mathcal{V}^\e \|_{L^2_tL^\infty_x}=  \| \nabla_y^2 \mathcal{V}^\e_{\Omega} \|_{L^2_tL^\infty_x}.
\]
The short-time Strichartz estimates available for solutions $u^\eps$ without rotation therefore directly transfer to $u^\eps_\Omega$, a fact which has already been recognized in \cite[Remark 2.2]{CDH}.

%~~~~~~~~~~~~~~~~~~~~~~~~~~~~~ ~~~~~~~~~~~~

\bibliographystyle{amsplain}

\begin{thebibliography}{99}
	
	
	\bibitem{AnMaSp} P. Antonelli, D. Marahrens, and C. Sparber, {On the Cauchy problem for nonlinear Schr\"odinger equations with rotation}. 
	{\it Discrete Contin. Dyn. Syst.}, {\bf 32} (2012), no.3, 703--715.
	
	\bibitem{ArNeSp} J. Arbunich, I. Nenciu, and C. Sparber, {Stability and instability properties of rotating Bose-Einstein condensates}. 
	{\it Lett. Math. Phys.}, {\bf 109} (2019), no. 6, 1415--1432.
	
	\bibitem{CC} P. Cao, R. Carles, {Semiclassical wave-packet dynamics for Hartree equations}. {\it Rev. Math. Phys.}, {\bf 23} (2011),  933--967
	
	\bibitem{Carles} R. Carles,  {Nonlinear Schr\"odinger equation with time dependent potential}. {\it Commun. Math. Sci.}, {\bf 9} (2011), no.4, 937--964.
	
	\bibitem{Car} R. Carles, {Semi-classical analysis for nonlinear Schr\"odinger equations—WKB analysis, focal points, coherent states}, World Scientific Publishing, 2021.
	
	\bibitem{CDH} R. Carles, V. D. Dinh, and H. Hajaiej, {On stability of rotational 2D binary Bose-Einstein condensates.} {\it Ann. Fac. Sci. Toulouse Math. (6)}, {\bf 32} (2023), no. 1, 81--124.
	
        \bibitem{CaFe} R. Carles and C. Fermanian-Kammerer,  {Nonlinear coherent states and Ehrenfest time for Schr\"odinger equation}. {\it Comm. Math. Phys.}, {\bf 301} (2011), 443--472.
        
        \bibitem{CoRo} M. Combescure, and D. Robert, {Semiclassical spreading of quantum wave packets and applications near unstable fixed points of the classical flow}. 
        {\it Asymptot. Anal.}, {\bf 14} (1997), no. 4, 377--404.
        
        \bibitem{Dam} M. D'Amico, {Fundamental solutions and smoothness in Schr\"odinger problems with applications to quantum fluids}. Ph.D. thesis, Gran Sasso Science Institute, 2017.
        
        \bibitem{Fu} D. Fujiwara,  {Remarks on the convergence of the Feynman path integrals}. {\it Duke Math. J.}, {\bf 47} (1980), no.3, 559--600.
	
	\bibitem{Ha} G. A. Hagedorn,  {Semiclassical quantum mechanics. I. The $\hslash \to 0$ limit for coherent states}. {\it Comm. Math. Phys.}, {\bf 71} (1980), no.1, 77--93.
	
	\bibitem{He} E. J. Heller, {Frozen Gaussians: a very simple semiclassical approximation}. {\it J. Chem. Phys.}, {\bf 75} (1981), no. 6, 2923--2931.
	
	\bibitem{KeTa} M. Keel and T. Tao, {Endpoint Strichartz estimates}, {\it Amer. J. Math.}, {\bf 120} (1998), no. 5, 955--980.
	
	\bibitem{KiOh} N. King and T. Ohsawa, {Hamiltonian dynamics of semiclassical Gaussian wave-packets in electromagnetic potentials}, {\it J. Phys. A}, {\bf 53} (2020), no. 10, 105201, 17pp.
	
	\bibitem{Ki} H. Kitada, {On a construction of the fundamental solution for Schr\"odinger equations}. {\it J. Fac. Sci. Univ. Tokyo Sec. IA Math.}, {\bf 27} (1980), 193--226.
		
	\bibitem{LaLu} C. Lasser and C. Lubich, {Computing quantum dynamics in the semiclassical regime}, {\it Acta Numer.}, {\bf 29} (2020), 229--401.
	
	\bibitem{RoCo} D. Robert and M. Combescure, {Coherent states and applications in mathematical physics}, Theoretical and Mathematical Physics, Springer Verlag, 2021.
	
	\bibitem{SVT} R. Schubert, R. O. Vallejos, and F. Toscano, {How do wave packets spread? Time evolution on Ehrenfest time scales}. {\it J. Phys. A: Math. Theor.}, {\bf 45} (2012), 215307, 27pp.
	
	\bibitem{AS} A. Schulze, {Existence of axially symmetric solutions to the Vlasov-Poisson system depending on Jacobi's}. {\it Commun. Math. Sci.}, {\bf 6} (2008), no.3, 711--727.
	
	\bibitem{Zhou} Z. Zhou, {Numerical approximation of the Schr\"odinger equation with the electromagnetic field by the Hagedorn wave-packets}, {\it J. Comput. Phys.}, {\bf 272} (2014), 386--407.

\end{thebibliography}

\end{document}